\documentclass[11pt]{amsart}

\usepackage[T1]{fontenc}
\usepackage[english]{babel}
\usepackage{amsmath,amssymb,amsthm}
\usepackage{mathtools}
\usepackage{enumitem}
\usepackage{mymacros}
\usepackage{hyperref}
\usepackage{xcolor}
\usepackage{bigints}
\usepackage{textgreek}
\def\CB{\color{black} }

\let\originalleft\left
\let\originalright\right
\renewcommand{\left}{\mathopen{}\mathclose\bgroup\originalleft}
\renewcommand{\right}{\aftergroup\egroup\originalright}

\title[Generalized Hilbert matrix operators on Bergman spaces]{Generalized Hilbert matrix operators acting on Bergman spaces}

\begin{document}

\author[ ]{C. Bellavita$^1$ }

\author[ ]{V. Daskalogiannis$^{2, 3}$}

\author[ ]{S. Miihkinen$^4$}

\author[ ]{D. Norrbo$^4$}

\author[ ]{G. Stylogiannis$^2$}

\author[ ]{J. Virtanen$^4$}

\email{\newline carlobellavita@ub.edu \\  vdaskalo@math.auth.gr \\  s.j.miihkinen@reading.ac.uk    \newline d.norrbo@reading.ac.uk\\ stylog@math.auth.gr\\  j.a.virtanen@reading.ac.uk }

\address{$^1$Departament of Matem\'atica i Inform\'atica, Universitat de Barcelona, Gran Via 585, 08007 Barcelona, Spain.}
\address{$^2$Department of Mathematics, Aristotle University of Thessaloniki, 54124, Thessaloniki, Greece.}
\address{$^3$Division of Science and Technology, American College of Thessaloniki, 17 V. Sevenidi St., 55535, Pylea, Greece.}
\address{$^4$Department of Mathematics and Statistics, School of Mathematical and Physical Sciences, University of Reading, Whiteknights, PO Box 220, Reading RG6 6AX, UK.}

\thanks{}




\begin{abstract}
    In this article we study the generalized Hilbert matrix operator $\Gamma_\mu$ acting on the Bergman spaces $A^p$ of the unit disc for $1\leq p<\infty$. In particular, we characterize the measures $\mu$ for which the operator $\Gamma_\mu$ is bounded and  we provide estimates of its operator norm. Finally, we also describe when $\Gamma_\mu$ is compact by computing  its essential norm.
\end{abstract}

\maketitle

\section{Introduction}
\noindent
Let $\mu$ be a finite, probability  Borel measure  on $[0,1)$.
We consider the infinite matrix
$$
\tilde{\Gamma}_{\mu}=\left(%
\begin{array}{ccccc}
            \gamma_{00} & \gamma_{01}  & \gamma_{02}  & \dots  \\
\gamma_{10} & \gamma_{11}  & \gamma_{12}  & \dots  \\
  \gamma_{20} & \gamma_{21}  & \gamma_{22}  & \dots \\
  \vdots  & \vdots & \vdots & \ddots  \\
\end{array}%
\right)
$$
with entries
$$
\gamma_{n k}=\binom{n+k}{k}\int_0^1t^k(1-t)^n\,d\mu(t). 
$$
The matrix $\tilde{\Gamma}_\mu$ is related to the classical Hausdorff matrix $\mathcal{K}_\mu$ induced by the moment sequence $\{\mu_n\}$, that is, for $n=0,1, \dots$,
\[
\mu_n=\int_0^1t^n\,d\mu(t).
\]
 Indeed
 \begin{equation}
\mathcal{K}_{\mu}= \left(
\begin{matrix}
c_{00}& 0 & 0&  \dots     \\
c_{10}& c_{11} &0& \dots     \\
c_{20} & c_{21} & c_{22} & \dots    \\
 \vdots  & \vdots & \vdots & \ddots  \\
\end{matrix}
\right) 
\notag
\end{equation}
with entries $c_{nk}$  given by
$$
c_{n k}=\binom{n}{k}\int_0^1t^k(1-t)^{n-k}\,d\mu(t),\;0\leq k\leq n.
$$
The matrix $\tilde{\Gamma}_\mu$ is obtained by shifting the entries of the $k$-th column of $\mathcal{K}_\mu$, $k$-places up.
More precisely, with respect to the standard basis $\{e_n\}_{n\geq 0}$, we note that $\tilde{\Gamma}_\mu$ is related to $\mathcal{K}_\mu$ through the algebraic relation 
\[
\tilde{\Gamma}_\mu(e_n)={S^*}^n\circ \mathcal{K}_\mu(e_n),
\]
where $S^*(e_0)=0$ and $S^{*}(e_n)=e_{n-1}$, for $n\geq 1$.
The Hausdorff matrices have been studied on spaces of analytic functions, see for example
 \cite{Galanopoulos2001}, \cite{Galanopoulos2006} and \cite{LIFLYAND2007}. 
 
\vspace{11 pt}\noindent 
We will study the action of the matrix  $\tilde{\Gamma}_\mu$ as an operator on spaces of analytic functions. In \cite{Bellavita2023}, the authors studied the action of $\tilde{\Gamma}_\mu$ in the Hardy spaces $H^p$ for $1\leq p<\infty$ and they characterized the measures $\mu$ for which the operator $\tilde{\Gamma}_\mu$ is bounded. The matrix $\tilde{\Gamma}_\mu$ acts  on the sequence of the Taylor coefficients of the analytic function $f(z)=\sum_{k=0}^{\infty} a_kz^k$ as follows:
\begin{equation*}
\tilde{\Gamma}_\mu (f)(z) = \sum_{n= 0}^{\infty}\left(\sum_{k= 0}^{\infty}a_k\, \binom{n+k}{k}\int_0^1
t^k(1-t)^n \,d\mu(t)\right)\,z^n.
\end{equation*}
We prove that $\tilde{\Gamma}_\mu (f)(z)$ has an equivalent integral representation in $A^p$ for $1\leq p<\infty$, i.e.
\begin{equation}\label{E:integral representation}
    \Gamma_{\mu}(f)(z) =\int_{0}^{1}
f(\varphi_{t}(z))w_{t}(z) \,d\mu(t) =
\int_{0}^{1}T_{t}(f) \,d\mu(t),
\end{equation}
where $T_t (f)=w_t\cdot  f\circ \varphi_t$ is a weighted composition operator with
$$
\varphi_{t}(z) =\frac{t}{1+(t-1)z} \quad
\text{and}\quad w_{t}(z) =\frac{1}{1 + (t-1)z}.
$$
We note that 
$$ 
\varphi_{t}(\mathbb{D})=D\left( \frac{1}{2-t},\;\frac{1-t}{2-t}\right),
$$
is the open disc centered at $1/(2-t)$ with radius $(1-t)/(2-t)$. In particular, for every $0<t<1\,$, $\,\varphi_{t}(\mathbb{D})\subset \mathbb{D}$ with $\overline{\varphi_{t}(\mathbb{D})}\cap \partial \mathbb D=\{1\}$.

\vspace{11 pt}\noindent 
If $\mu$ is the Lebesgue measure, the matrix $\tilde{\Gamma}_\mu$ becomes the classical Hilbert matrix
\[
H=\left(
\begin{array}{ccccc}
            1 & \frac{1}{2}  & \frac{1}{3}  & \dots  \\ [4pt]
  \frac{1}{2} & \frac{1}{3}  & \frac{1}{4}  & \dots  \\ [4pt]
  \frac{1}{3} & \frac{1}{4}  & \frac{1}{5}  & \dots \\
  \vdots  & \vdots & \vdots & \ddots  \\
\end{array}
\right).
\]
\noindent 
The action of this operator has been widely studied in the Bergman spaces. 
In \cite{Diamantopoulos2004}, using the integral representation \eqref{E:integral representation}, Diamantopoulos showed that $H$ is bounded on $A^p$ if and only if $p>2$, providing a sharp upper bound when $p\geq 4$. A sharp lower bound for $\norm{H}_{A^p}$ was determined in \cite{Dostanic2008} for every $p>2$,  by considering appropriate test functions, hence the exact value of the norm for $p\geq 4$ was computed. In \cite{Bozin2018}, the authors estimated a sharp upper bound for the norm of $H$ on $A^p$, for $2<p<4$, applying some new estimates for the Beta function. In particular, for $p>2$, 
\begin{equation}\label{E:norm hilbert matrix}
\|H\|_{A^p}=\dfrac{\pi}{\sin\left(\frac{2\pi}{p}\right)}=\int_0^1 \frac{t^{2/p-1}}{(1-t)^{2/p}} \, dt .
\end{equation}
See also \cite{Lindstrom2021} where a simplified proof of \eqref{E:norm hilbert matrix} appears.

\vspace{11 pt}\noindent 
There are still several open questions related to the action of the classical Hilbert matrix operator on spaces of analytic functions. Although we know that $H$ is bounded in the standard weighted Bergman spaces $A^p_\alpha$ if and only if $ 1<\alpha+2<p$ \cite{Karapetrovic2018}, the exact value of $\|H\|_{A^p_\alpha}$ is known only for
\begin{equation*}
p\geq \frac{3\alpha}{4}+2+\sqrt{\left(\frac{3\alpha}{4}+2\right)^2-\frac{\alpha + 2}{2} } 
\end{equation*}
when $\alpha > 0$, see \cite{Dmitrovi2023}. We mention \cite{Bellavitasurvey} as a brief review of the latest updates on the study of the Hilbert matrix operator acting on different spaces of analytic functions and on sequence spaces.

\vspace{11 pt}
\noindent 
In this article, we focus on necessary and sufficient conditions for the continuity of the operators $\Gamma_\mu$ on $A^p$ with $1\leq p<\infty$.
In order to formulate our main result it will be convenient to introduce the following function:

\begin{small}
\[ 
\Theta_p(t):= 
\begin{cases}
    \dfrac{t^{\frac{2}{p}-1}}{(1-t)^{\frac{2}{p}}},\,\,\,\, &\text{if} \,\, 2<p<\infty;
    \\[0.2in]
 \ \sqrt{\int_0^1 \mu[0,t]\log\left(\frac{e}{t}\right)\,d\mu(t)+\left(\int_0^1\frac{1}{1-t}\,d\mu(t)\right)^{2}},   \,\,\,\, &\text{if} \,\, p=2;
    \\[0.2in]
   \dfrac{1}{(1-t)^{2/p}},\, \,\,\,&\text{if}\,\, 1\leq p<2 .
\end{cases} 
\]
\end{small} 
\noindent 
We use the following convention:
\[
 \mu[a,b]:= \int_a^b \,d\mu(s) 
:=  \int_{[a,b)} \, d\mu(s), \quad  0\leq a<b\leq 1. 
\]
\begin{thm}\label{main theom boundedness}
The operator $\Gamma_\mu$ is bounded on  $A^p, \;1\leq p<\infty$,  if and only if
\begin{equation*}
    \int_{0}^{1}\Theta_p(t)\,d\mu(t)<\infty.
\end{equation*}
In particular, there exist positive constants $A(p), B(p)$ depending only on $p$ such that 
\begin{equation}\label{E:extra}
    A(p)\int_0^1 \Theta_p(t)\,d\mu(t)\, \leq \,\norm{\Gamma_\mu}_{A^p \to A^p}\,\leq\, B(p) \int_{0}^{1} \Theta_p(t)\,d\mu(t). 
\end{equation}
When $p>2$ the constant $A(p)$ can be chosen equal to $1$ and when $p\geq 4$ the constant $B(p)$ can also be chosen equal to $1$.  Hence, when $p\geq4$
\[
\|\Gamma_\mu\|_{A^p\to A^p}=\int_{0}^{1}\dfrac{t^{\frac{2}{p}-1}}{(1-t)^{\frac{2}{p}}}\,d\mu(t).
\]
\end{thm}
\noindent
The expression of $\Theta_p(t)$ is different when $p=2$ and in Proposition \ref{Prop_counter} we show that the natural condition coming from Lemma \ref{L:estimates T_t}, that is
\begin{equation}
   \int_0^1 \dfrac{\sqrt{\log(e/t)}}{1-t}\,d\mu(t)<\infty,
    \end{equation}
    is not necessary for the boundedness of $\Gamma_\mu$ on $A^2$.

\vspace{11 pt}
\noindent
 Notice that, even when $2<p<4$, the lower bound of the norm is the optimal bound for which the inequality \eqref{E:extra} is valid for all measures $\mu$. Indeed the lower bound in \eqref{E:extra} corresponds to the precise value of the norm given in \eqref{E:norm hilbert matrix}.
On the other hand, from \cite[Corollary 3.2]{Liu2015}, we know that when $p\to 2^+$
$$
\sup_{a \in \mathbb D}\|\Gamma_{\delta_a}(1)\|_{A^p}\Big/ \int_{0}^1\Theta_p(t)  \, d\delta_a(t)>1,
$$
where $\delta_a$ is a Dirac point measure at $a \in (0,1)$.
This implies that for specific measures and particular values of $p$, the constant $A(p)$ could be chosen bigger than $1$.

\vspace{11 pt}
\noindent 
Furthermore, we consider compactness and complete continuity.
We recall that an operator $T$ on a Banach space $X$ is compact if, for any bounded sequence $\{x_n\}$ in X, the sequence
$\{T(x_n)\}$ contains a convergent subsequence.
Moreover, an operator $T$ is completely continuous on $X$ if for any  weakly convergent sequence $\{x_n\}$ in $X$, the sequence
$\{T(x_n)\}$ converges in norm. In general, every compact operator is completely continuous, however the converse could be false when $X$ is non-reflexive.

\vspace{11 pt}
\noindent 
In order to prove that an operator $T$ is non-compact, it is enough proving that its essential norm $\|T\|_{e,X}$ is non-zero.
We recall that
$$
\| T\|_{e,X}=\inf \left\lbrace \|T-K\|_X \text{ where } K \text{ is a compact operator in } X\right\rbrace. 
$$
 It is  clear that $\|T\|_{X} \geq \|T\|_{e,X}$.

\begin{thm}\label{Compact p>1}
Let $\Gamma_{\mu}$ be bounded in $A^p$, $1<p<\infty$. Then,

\[
\|\Gamma_\mu\|_{e,A^p}=\int_{0}^1\frac{t^{2/p-1}}{(1-t)^{2/p}}\,d\mu(t).
\]
\end{thm}
\noindent 
It is clear from the theorem above that $\Gamma_{\mu}$ is never compact on $A^p$ if $2\leq p<\infty$. On the other hand, if $1< p<2$,  $\Gamma_\mu$ is compact in $A^p$ if and only if $\mu(t)=\delta_0(t)$, that is, if $\mu(t)$ is the Dirac point mass $0$. Indeed, in this case,
\[
 \Gamma_{\delta_0}(f)(z)=f(0)\dfrac{1}{1-z}
 \]
 which is clearly compact as a rank one  operator.

\begin{thm}\label{Compact p=1}
Let $\Gamma_{\mu}$ be bounded on $A^1$. $\Gamma_\mu$ is compact if and only if $\mu(t)=\delta_0(t)$. Nevertheless, for every probability measure $\mu$ such that $\Gamma_{\mu}$ is bounded on $A^1$, $\Gamma_\mu$ is completely continuous.
\end{thm}

\noindent
The rest of the paper is organized as follows: In Section 2,  we recall the classical properties of the Bergman space and we prove that  $\Gamma_\mu(f)$ is a well-defined analytic function and that the action of $\Gamma_ \mu$ coincides with that of $\tilde{\Gamma}_\mu$ on $A^p$.  We also provide an upper estimate for $\|T_t\|_{A^p}$ when $1\leq p<\infty$. 
We split the proof of Theorem \ref{main theom boundedness} into Sections 3 and 4. In Section 3, we prove the boundedness of $\Gamma_\mu$ in $A^p$ when $1\leq p<\infty$ but $p\neq 2$. In Section 4, we fix our attention on the case $p=2$ and we also provide some conditions for boundedness of which some are sufficient and some are necessary. In Section 5, we deal with the compactness and we prove Theorems \ref{Compact p>1} and \ref{Compact p=1}.

\vspace{11 pt}
\noindent
 Before we dive into calculations, we first clarify the notation that we will use in the following sections. 
With $\mathbb D$ we refer to the unit disc in the complex plane $\mathbb C$ and with $\mathbb T$ its boundary $\partial \mathbb D$. 
Given a set $M$, by $\chi_M$ we denote the characteristic function associated to the set $M$. 
We use the expression $\|f\|_{X}$ to denote the norm of an element $f \in X$. Moreover, if $T$ is an operator from the space $X$ to $X$ by $\|T\|_X$ we mean its operator norm, that is 
$$
\|T\|_X=\sup_{\|f\|_X=1}\|T(f)\|_X.
$$
 Even if the two notations coincide, this should not cause confusion in this context.
Finally, by the expressions $f \lesssim g$ and $g \gtrsim f$, we mean that there exists a constant $C>0$ such that  
$$
f\leq C g. 
$$
If both \(f\lesssim g\) and \(f\gtrsim g\) hold, we write \(f\sim g\). We also highlight that by the capital letter \(C\), we denote constants whose values may change every time they appear. 
\vspace{11 pt}
\section{Preliminaries}
\noindent 
First of all, for the sake of completeness, we recall the properties of the Bergman spaces $A^p$ that will be used in the rest of the paper.
For  $1\leq p<\infty$, the
Bergman space $A^p$ consists of all the analytic functions in the unit disc such that
\[
\norm{f}_{A^p}\,=\,\left(\int_{\mathbb D}|f(z)|^p \, dA(z)\right)^{1/p} <\infty,
\]
where $dA(z)={dx dy}/{\pi}$ is the normalized Lebesgue area measure.
We recall that, if $f\in A^p$ with $1\leq p<\infty$, the  growth estimates 
\begin{equation}\label{growth}
\vert f(z)\vert
\,\leq\,
 \left(\frac{1}{1-\vert z\vert^2}\right)^{\frac{2}{p}}\norm{f} _{A^p},\quad z\in \mathbb{D}
\end{equation}
and for some independent \(C>0\)  
\begin{equation}\label{growthderivative}
\vert f'(z)\vert
\,\leq\, C
 \left(\frac{1}{1-\vert z\vert^2}\right)^{\frac{2}{p}+1}\norm{f} _{A^p},\quad z\in \mathbb{D}
\end{equation}
hold, see \cite[p.~755]{Vukotic1993} and \cite[p.~338]{Luecking1993} respectively. Moreover, if $f(z)=\sum_{n\geq 0}a_nz^n$, then
\begin{equation}\label{norm bergman series}
    \|f\|_{A^2}^2=\sum_{n\geq 0}\dfrac{|a_n|^2}{n+1},
\end{equation}
see \cite[p.~11]{Duren2004}.  We remark that the Taylor partial sums of $f$, that is, the polynomials
\[
S_n(f)(z)=\sum_{i=0}^n\hat{f}(i)z^i
\]
 with $n \in \mathbb N$, converge in $A^p$ norm to 
\[
f(z)=\sum_{i=0}^\infty\hat{f}(i)z^i
\]
when \(1<p<\infty\), see \cite[p.~31, Theorem 4]{Duren2004}. For more information about the Bergman spaces we refer to the classical monographs \cite{Duren2004} and \cite{hedenmalm2012}.

\vspace{11 pt}
\noindent 
The first thing that needs to be verified is that the integral \eqref{E:integral representation} involved in the definition of  $\Gamma_\mu(f)$ is a well-defined analytic function in $\mathbb D$. We follow the reasoning  of \cite{Diamantopoulos2004}.

\begin{prop}\label{P:welldefinition}
For \(1\leq p<\infty\), let 
\begin{equation*}
\psi_p=
  \int_{0}^1 \frac{1}{(1-t)^{2/p}}\,d\mu(t).
\end{equation*}
If $\psi_p<\infty$, then for every $f \in A^p$,  $\Gamma_\mu(f)$ is a well-defined analytic function.
\end{prop}
\begin{proof}
We first prove  that for every fixed $z \in \mathbb D$,  $\Gamma_\mu(f)(z)$ is well-defined. Indeed, due to \eqref{growth}, we have
\begin{align*}
    |{\Gamma_\mu(f)(z)}|\leq& \int_0^1 \frac{1}{|{1-(1-t)z}|}|f(\varphi_t(z))|\,d\mu(t)\\
    \leq & \int_0^1 \frac{1}{|{1-(1-t)z}|^{1-2/p}}  \frac{1}{(1-t)^{2/p}}\,d\mu(t)\cdot\frac{\|f\|_{A^p} }{(1-|z|)^{2/p}}\\
    \leq & \frac{\|f\|_{A^p} }{(1-|z|)^{2/p+1}}  \int_0^1  \frac{1}{(1-t)^{2/p}}\,d\mu(t).
\end{align*}    
\noindent
In order to prove that $\Gamma_\mu(f)$ is analytic in $  \mathbb D$, we show that there exists a sequence of analytic functions which converges to  $\Gamma_\mu(f)$ uniformly on every compact subset of $\mathbb D$.  
Let $\{P_n\}_n$ be a family of polynomials such that  
 $$
 \lim_{n \to \infty }\|f-P_n\|_{A^p}=0,
 $$
 see \cite[p.~30, Theorems 3 and 4]{Duren2004}. 
 It is clear that for every $n\in \mathbb{N}$, $\Gamma_\mu(P_n)$ is analytic in $\mathbb{D}$. Therefore, for every $z$ in a compact set $K\subset \mathbb{D}$,
 \begin{equation}
 \label{E:welldefinition}
 \begin{split}
\big| \Gamma_\mu(f)(z)-\Gamma_\mu(P_n)(z)\big|&\leq \|f-P_n\|_{A^p} \frac{ \psi_p }{(1-|z|)^{2/p+1}}\\
&\leq  \|f-P_n\|_{A^p}\frac{ \psi_p }{\text{dist}(K,\partial \mathbb D)^3},
 \end{split}
 \end{equation}
 which concludes the proof of the proposition. \\
\end{proof}

\noindent 
The fact that in $A^p$ the operators  $\Gamma_\mu$ and  $\tilde{\Gamma}_\mu$ coincide requires some standard estimates. For the sake of completeness, we will write them down.

\begin{prop}
Let $1\leq  p<\infty $ and $\psi_p<\infty$, where $\psi_p$ is defined as in Proposition \ref{P:welldefinition}. If $f \in A^p$, then  
$$
\Gamma_\mu(f)=\tilde{\Gamma}_\mu(f).
$$
\end{prop}
\begin{proof}
It is clear that for every analytic polynomial
$$
\Gamma_\mu(P_n)=\tilde{\Gamma}_\mu(P_n).
$$
 We first consider the case $1<p<\infty$ and $f(z)=\sum_{k\geq 0} a_k z^k \in A^p$. Then for every $z \in \mathbb{D}$, by the proof of Proposition \ref{P:welldefinition},
\begin{align*}
    \lim_{n\to \infty} |{\Gamma}_\mu(f)(z)-\tilde{\Gamma}_\mu(S_n)(z)|&= \lim_{n\to \infty}  |{\Gamma}_\mu(f)(z)-{\Gamma}_\mu(S_n)(z)|=0,
\end{align*}
where $S_n$ is the Taylor partial sum.
Since ${\Gamma}_\mu(f)$ and $\tilde{\Gamma}_\mu(S_n)$ are analytic functions, the Taylor coefficients coincide, that is, for every $\ell \in \mathbb{N}$,
$$
\widehat{\Gamma_\mu(f)}(\ell)=\lim_{n \to \infty} \sum_{k= 0}^{n}a_k\, \binom{\ell+k}{k}\int_0^1
t^k(1-t)^\ell \,d\mu(t),
$$
which concludes the first part of the proof.

\vspace{11 pt}\noindent 
For $p=1$ and  $f(z)=\sum_{k\geq 0} a_k z^k \in A^1$, from \eqref{growth} and \eqref{growthderivative}, we note that
\begin{align*}
    |a_0|&\leq \|f\|_{A^1}\, \text{ and }\,  |a_1|\leq C\|f\|_{A^1}.
\end{align*}
Moreover, by the Cauchy formula, for every $r\geq 1/2$ and $n\geq 2$, we note that
\begin{align*}
    |a_n|&\leq r^{-n} \int_{0}^{2\pi}|f(re^{i\theta})|\dfrac{d\theta}{2\pi}\\
    &=r^{-n}(1-r)^{-1}\left\lbrace  (1-r)\int_{0}^{2\pi}|f(re^{i\theta})|\dfrac{d\theta}{2\pi}\right\rbrace\leq r^{-n}(1-r)^{-1} 2 \|f\|_{A^1}.
\end{align*}
In particular, by choosing $r=1-1/n$, we obtain that
$$
   |a_n| \leq 2\left( 1-\dfrac{1}{n}\right)^{-n} (n+1)\|f\|_{A^1}\leq 8 (n+1)\|f\|_{A^1}.
$$
Therefore, for every $z \in \mathbb{D}$, 
\begin{align*}
    |\tilde{\Gamma}_\mu(f)(z)|&= \left\vert\sum_{n= 0}^{\infty}\left(\sum_{k= 0}^{\infty}a_k\, \binom{n+k}{k}\int_0^1
t^k(1-t)^n \,d\mu(t)\right)\,z^n \right\vert\\
&\leq \sum_{n= 0}^{\infty}\left(\sum_{k= 0}^{\infty}|a_k|\, \binom{n+k}{k}\int_0^1
t^k(1-t)^n \,d\mu(t)\right)\,|z|^n \\
&\leq C \sum_{n= 0}^{\infty}\left(\sum_{k= 0}^{\infty}(k+1)\, \binom{n+k}{k}\int_0^1
t^k(1-t)^n \,d\mu(t)\right)\,|z|^n  \cdot \|f\|_{A^1}.
\end{align*}
Thus, since 
\begin{small}
$$
\sum_{n= 0}^{\infty} \binom{n+k}{k}
(1-t)^n |z|^n t^k= \left( \dfrac{t}{1-(1-t)|z|}\right)^k \dfrac{1}{1-(1-t)|z|}=\varphi_t(|z|)^k w_t(|z|),
$$
\end{small}
we have that  
\begin{align*}
|\tilde{\Gamma}_\mu(f)(z)|&\lesssim  \int_0^1 \left(\sum_{k= 0}^{\infty}(k+1)\varphi_t(|z|)^k  \right)
w_t(|z|) \,d\mu(t) \cdot \|f\|_{A^1}\\
&=\int_{0}^1\frac{1}{(1-\varphi_t(|z|))^{2}} w_t(|z|)\,d\mu(t)\cdot \|f\|_{A^1} \\
&\leq \int_{0}^1\frac{1}{(1-t)^{2}}\,d\mu(t)\cdot \frac{\|f\|_{A^1}}{(1-|z|)^{2}}\\
&\leq \psi_1 \cdot \frac{\|f\|_{A^1}}{(1-|z|)^{2}}.
\end{align*}

\noindent
Therefore, if $P_n \to f $ in $A^1$, due to \eqref{E:welldefinition}, for every $z \in \mathbb D$, we have
$$
\Gamma_\mu(f)(z)=\lim_{n \to \infty}\Gamma_\mu(P_n)(z)=\lim_{n \to \infty}\tilde{\Gamma}_\mu(P_n)(z)=\tilde{\Gamma}_\mu(f)(z),
$$
which concludes the proof for $p=1$.\\
\end{proof}
\noindent The weighted composition operator
$$
T_{t}(f)(z) =\frac{1}{(t - 1)z + 1}f\left(\frac{t}{(t-1)z + 1}\right)
$$
has been intensively studied in $A^p$ in relation to the classical Hilbert matrix operator. In specific, Diamantopoulos \cite{Diamantopoulos2004} estimated its norm when $2<p<\infty$. We extend this result to the case $1\leq p \leq 2$.
\begin{lem}\label{L:estimates T_t}
Let $1 \leq p <\infty$ and $f\in A^p$. Then
\begin{equation}\label{E:norm T_t}
    \norm{T_t(f)}_{A^p}\leq C(p)\; \Psi_p(t)\;\norm{f}_{A^p}\
\end{equation}
where $C(p)$ is a positive constant depending only on $p$ and 
\begin{equation}
\Psi_p(t)= 
\begin{cases}
    \dfrac{t^{\frac{2}{p}-1}}{(1-t)^{\frac{2}{p}}},\,\,\,\, &\text{if} \,\, 2<p<\infty;
    \\[0.2in]
   \dfrac{\sqrt{\log\left(e/t\right)}}{1-t},\, \,\,\,&\text{if}\,\, p=2;
       \\[0.2in]
   \dfrac{1}{(1-t)^{2/p}},\, \,\,\,&\text{if}\,\, 1\leq p<2 .
\end{cases} 
\end{equation}
Moreover, if $p\geq 4$, we have $C(p)=1$.
\end{lem}
\begin{proof}
We recall that in \cite[Lemma 2]{Diamantopoulos2004}, the constant appearing in \eqref{E:norm T_t} was estimated for $2<p<\infty$ as 
\[ 
C(p)=
\begin{cases}
 1, \,\,\,\, &\text{if} \,\, 4\leq p<\infty;
\\[0.2in]
\left( \frac{2^{7-p}}{9(p-2)}+2^{4-p}\right)^{1/p},\, \,\,\,&\text{if}\,\, 2< p<4 .
\end{cases}
\]
Let us consider $1< p\leq 2$.
For every $z \in \mathbb D$, $f(z)=f(z)-f(0)+f(0)$. Therefore
\begin{equation}\label{E1}
    \|T_t(f)\|_{A^p}^{p}\leq 2^{p-1}\big( \|T_t(f-f(0))\|_{A^p}^{p}+\|T_t\left(f(0)\right)\|_{A^p}^{p}\big) .
\end{equation}
We start by estimating the second term. We note that
\begin{align*}
    \|T_t\left(f(0)\right)\|_{A^p}^p =&|f(0)|^p\int_{\mathbb D}\frac{1}{|1-(1-t)z|^p} \, dA(z).
\end{align*}
Applying Forelli-Rudin estimates, see \cite[Lemma 3.10]{Zhu1990}, and \eqref{growth}, we get that there is some constant $B_p$ depending only on $p$ such that 
$$
\|T_t\left(f(0)\right)\|_{A^p}^p\leq \|f\|_{A^p}^p \cdot
\begin{cases}
 B_p, \, \,\,\, &\text{if} \,\, 1\leq p<2;
\\[0.2in]
    B_2\log(e/t),\, \,\,\,&\text{if}\,\, p=2 .
\end{cases} 
$$
For the first term in \eqref{E1}, a change of variables yields
\begin{align*}
    \|T_t(f-f(0))\|_{A^p}^p 
    &=\frac{t^{2-p}}{(1-t)^2}\int_{\varphi_t(\mathbb D)} |w|^{2p-4}\left |S^*(f)(w)\right|^p \, dA(w)\\
    \leq& \frac{t^{2-p}}{(1-t)^2} \int_{\mathbb D} |w|^{2p-4}\left |S^*(f)(w)\right|^p \, dA(w)\\
    =& \frac{t^{2-p}}{(1-t)^2}\Big( \int_{|w|\leq 1/2}|w|^{2p-4}\left |S^*(f)(w)\right|^p \, dA(w) \;+\\
    &\qquad + \int_{1/2\leq |w|< 1}|w|^{2p-4}\left |S^*(f)(w)\right|^p \, dA(w)\;\Big)\\
    =&\frac{t^{2-p}}{(1-t)^2}\left(I+II \right),
\end{align*}
where $S^*$ is the backward shift and  we have used the fact that
$$
 |\varphi_t(z)|^4 =\frac{t^{2}}{(1-t)^2}|\varphi_t'(z)|^2.
$$
Then, by using \eqref{growth}, we have 
\begin{align*}
I&\leq \int_{|w|\leq 1/2}\frac{|w|^{2p-4}}{(1-|w|^2)^2}\, dA(w) \|S^*(f)\|_{A^p}^p\\
&\leq \frac{16}{9}\int_{|w|\leq 1/2}|w|^{2p-4} \, dA(w)\, \|S^*\|_{A^p}^p \,\|f\|_{A^p}^p\\
&\leq\frac{2^{(6-2p)}}{9}\frac{1}{p-1}\|S^*\|_{A^p}^p   \|f\|_{A^p}^p
\end{align*}
and
\begin{align*}
II&\leq\int_{ 1/2 \leq |w|<1}\left( \frac{1}{2}\right)^{2p-4} |S^*(f)(w)|^p \, dA(w)\leq 2^{4-2p}\|S^*\|_{A^p}^p  \|f\|_{A^p}^p . 
\end{align*}
Thus 
$$
 \|T_t(f-f(0))\|_{A^p}^{p}\leq \dfrac{2^{5-2p}}{p-1}\|S^*\|_{A^p}^p  \|f\|_{A^p}^p . 
$$
In the case $p=1$, for every $z\in \mathbb D$, we consider
$$
f(z)=f(0)+zf'(0)+z^2\cdot {S^{*}}^2(f)(z),
$$
where ${S^*}^2=S^*\circ S^*$.
Therefore
\begin{equation}\label{E1p=1}
   \|T_t(f)\|_{A^1}\leq  \|T_t(z^2 \cdot {{S^*}^2}(f))\|_{A^1}
   +\|T_t\left(f(0)\right)\|_{A^1}+\|T_t(z\cdot f'(0))\|_{A^1} .   
\end{equation}
We have already estimated the second term. For the third one, by using \eqref{growthderivative}, we note that
\begin{align*}
    \|T_t(zf'(0))\|_{A^1} =&|f'(0)| \cdot t\int_{\mathbb D}\frac{1}{|1-(1-t)z|^{2}}\, dA(z)\\
    \leq& C B_2  \cdot [t \log(e/t)]\cdot\|f\|_{A^1}.
\end{align*}
Finally, with computations similar to those done for $1<p<2$, we have that
\begin{align*}
 \|T_t(z^2\cdot {S^{*}}^2(f))\|_{A^1}=&\frac{1}{t} \int_{\mathbb D} \left|\frac{t}{1-(1-t)z}\right|^{3} \left |{S^{*}}^2(f)\left(\frac{t}{1-(1-t)z}\right)\right| \, dA(z)\\
 =& \frac{t}{(1-t)^2}\Big( \int_{|w|\leq 1/2}|w|^{3-4}\left |{S^{*}}^2(f)(w)\right|\, dA(w) \;+\\
    &\qquad + \int_{1/2\leq |w|< 1}|w|^{3-4}\left |{S^{*}}^2(f)(w)\right|\, dA(w)\;\Big)\\
    \leq&\frac{5t}{(1-t)^2} \|S^*\|_{A^1}^2\|f\|_{A^1} .
\end{align*}
Hence, by using \eqref{E1p=1}, if $p=1$, we have 
\begin{equation*}
    \begin{split}
\|T_t(f)\|_{A^1} &\leq \left( B_1+CB_2t\log\left(e/t\right) +5 \|S^*\|_{A^1}^2 \frac{t}{(1-t)^2}\right) \|f\|_{A^1} \\
&\leq C\frac{1}{(1-t)^2} \|f\|_{A^1}.
\end{split}
\end{equation*}
By using \eqref{E1}, if $1<p<2$,
\begin{equation*}
    \begin{split}
\|T_t(f)\|_{A^p}^p &\leq 2^{p-1} \left( B_p+\frac{2^{5-2p}}{p-1}\|S^*\|_{A^p}^p \frac{t^{2-p}}{(1-t)^2}\right) \|f\|_{A^p}^p\ \\
&\leq C \frac{1}{(1-t)^2}  \|f\|_{A^p}^p
\end{split}
\end{equation*}
If $p=2$,
\begin{equation*}
    \begin{split}
\|T_t(f)\|_{A^2}^2 &\leq 2 \left( B_2\log(e/t)+ \|S^*\|_{A^2}^2 \frac{1}{(1-t)^2}\right) \|f\|^2_{A^2}\ \\
&\leq C \frac{\log(e/t)}{(1-t)^2}  \|f\|^2_{A^2} .
\end{split}
\end{equation*}
The proof of the Lemma is now complete.\\
\end{proof}
\vspace{11 pt}
\section{Proof of boundedness of $\Gamma_\mu$: $p\neq 2$}
\noindent
First of all, we note that $\Gamma_\mu$ is never contractive.
Indeed, if $ \Gamma_\mu$ is bounded in $A^p$, then, by \eqref{growth},
\begin{equation*}
  1=\Gamma_\mu(1)(0)\leq\|\Gamma_\mu(1)\|_{A^p}\leq \|\Gamma_\mu\|_{A^p}. 
\end{equation*}

\noindent 
We provide the estimate from below of $\|\Gamma_\mu\|_{A^p}$ in a separate lemma, which will also be used in the computation of the essential norm of $\Gamma_\mu$.

\begin{lem}\label{lem-fond}
Let $1\leq p<\infty$. If $\Gamma_\mu$ is bounded in $A^p$, then \begin{equation*}
    \int_0^1 \dfrac{t^{2/p-1}}{(1-t)^{2/p}}\,d\mu(t)\leq\|\Gamma_\mu\|_{A^p}.
\end{equation*}
\end{lem}
\begin{proof}
Let $f_{a}(z)=1/(1-z)^{a}$ for $0<a<2/p$. We know that 
\begin{equation*}
    \lim_{a \to 2/p}\|f_a\|_{A^{p}}=\infty .   
\end{equation*}
We consider
\begin{equation*}
\begin{split}
    \Gamma_{\mu}(f_{a})(z)
    &=\int_{0}^{1}\frac{[1-(1-t)z]^{a-1}}{(1-t)^{a}}\,d\mu(t) \cdot f_{a}(z)=\Lambda_{a}(z)\cdot f_a(z) ,
\end{split}
\end{equation*}
where
\begin{equation*}
    \Lambda_{a}(z)=\int_{0}^{1}\frac{[1-(1-t)z]^{a-1}}{(1-t)^{a}}\,d\mu(t) .
\end{equation*}
Let $B_\varepsilon(1)$ be a circle of radius $\varepsilon$ centred at $1$. We have that 
\begin{align*}
\|\Gamma_\mu(f_a/\|f_a\|_{A^p})\|_{A^p}^p &\geq \int_{\mathbb D\cap B_\varepsilon(1)}\frac{|f_a(z)|^p}{\|f_a\|^p_{A^p}} \vert\Lambda_a(z)\vert^p \, dA(z)\\
&\geq \inf_{z\in \mathbb D\cap B_\varepsilon(1) } \vert\Lambda_a(z)\vert^p \cdot \int_{\mathbb D\cap B_\varepsilon(1)}\frac{|f_a(z)|^p}{\|f_a\|^p_{A^p}}\, dA(z)\\
&= \inf_{z\in \mathbb D\cap B_\varepsilon(1) } \vert\Lambda_a(z)\vert^p \cdot\left( 1-  \int_{\mathbb D\setminus B_\varepsilon(1)}\frac{|f_a(z)|^p}{\|f_a\|^p_{A^p}}\, dA(z)\right).
\end{align*}
Consequently, since $f_a/\|f_a\|_{A^p}$ has unitary norm in $A^p$ and 
$$
|f_a(z)|<\frac{1}{\varepsilon^a}\leq \frac{1}{\varepsilon^{2/p}} \text{ when } z \in \mathbb D\setminus B_\varepsilon(1) ,
$$ 
we find that 
\begin{align*}
    \|\Gamma_\mu\|^p_{A^p}\geq& \liminf_{\epsilon \to 0}\ \liminf_{a\to 2/p}\ \inf_{z\in \mathbb D\cap B_\varepsilon(1) } \vert\Lambda_a(z)\vert^p \cdot\left( 1-  \int_{\mathbb D\setminus B_\varepsilon(1)}\frac{|f_a(z)|^p}{\|f_a\|^p_{A^p}}\, dA(z)\right)\\
    =& \liminf_{\epsilon \to 0}\ \liminf_{a\to 2/p}\ \inf_{z\in \mathbb D\cap B_\varepsilon(1) } \vert\Lambda_a(z)\vert^p\\
     \geq &\left( \liminf_{\epsilon \to 0}\ \liminf_{a\to 2/p}\ \inf_{z\in \mathbb D\cap B_\varepsilon(1) } \Re \Lambda_a(z)\right)^p.
\end{align*}
Furthermore,
\begin{align}
    \inf_{z\in \mathbb D\cap B_\varepsilon(1) } \Re \Lambda_a(z)\geq \int_{0}^1\frac{1}{(1-t)^a} \inf_{z\in \mathbb D\cap B_\varepsilon(1) }\Re(1-(1-t)z)^{a-1}\,d\mu(t).
\end{align}
Let us consider $2\leq p<\infty$. Then, since $a<1$, we note that
\begin{align*}
&\inf_{z\in \mathbb D\cap B_\varepsilon(1) }\Re(1-(1-t)z)^{a-1}\\
&\quad =\inf_{z\in \mathbb D\cap B_\varepsilon(1) } \frac{1}{|1-(1-t){z}|^{2(1-a)}}\Re(1-(1-t)\overline{z})^{1-a}
\geq \frac{|t|^{1-\alpha}}{|t+\varepsilon|^{2(1-a)}} .
\end{align*}
On the other hand, if $1\leq p<2$, we consider $a>1$ and we get that
 \begin{align*}
\inf_{z\in \mathbb D\cap B_\varepsilon(1) }\Re(1-(1-t)z)^{a-1}
    &\geq t^{a-1} .
\end{align*}
Therefore, by taking the limits and applying Fatou's Lemma, we have that
\begin{equation*}
    \liminf_{\epsilon \to 0}\ \liminf_{a\to 2/p}\  \inf_{z\in \mathbb D\cap B_\varepsilon(1) }\Re \Lambda_a(z)\geq  \int_0^1 \dfrac{t^{2/p-1}}{(1-t)^{2/p}}\,d\mu(t) .
\end{equation*}
The statement of the lemma is now proved.\\
\end{proof}
\noindent 
We are now ready for the proof of Theorem \ref{main theom boundedness} when $p\neq 2$.

\begin{proof}[Proof of Theorem \ref{main theom boundedness} when $p \neq 2$]
We start with the sufficient condition. An application of the generalized Minkowski's inequality for the integrals together with Lemma \ref{L:estimates T_t} imply that for $1\leq p<\infty$
\begin{align*}
\nonumber \norm{\Gamma_\mu(f)}_{A^p} \, &\leq \,\int_0^1 \norm{T_t(f)}_{A^p} \,d\mu(t)
\leq B(p)\,\int_0^1\Theta_p(t)\,d\mu(t)\;\norm{f}_{A^p},
\end{align*}
 where $B(p)$ is equal to $C(p)$ from the proof of Lemma \ref{L:estimates T_t}.

\vspace{11 pt}\noindent 
We prove now the necessary condition.
We first consider $2<p<\infty$. Then Lemma \ref{lem-fond} implies that
\[
\|\Gamma_\mu\|_{A^p}\geq \int_0^1\Theta_p(t)\,d\mu(t) 
\]
from which the statement of the theorem follows. 

\vspace{11 pt}\noindent 
If $1\leq p<2$, we consider again Lemma \ref{lem-fond}, but we note that
\[
\begin{split}
\int_{0}^{1}\frac{1}{(1-t)^{2/p}} \,d\mu(t) & =\int_{0}^{1/2}\frac{1}{(1-t)^{2/p}}  \,d\mu(t) +\int_{1/2}^{1}\frac{1}{(1-t)^{2/p}} \,d\mu(t)
\\[0.1in]
&\leq 2^{2/p}\Gamma_\mu(1)(0) + 2^{2/p-1}\int_{1/2}^{1}\frac{t^{2/p-1}}{(1-t)^{2/p}}\,d\mu(t)\\
&\leq \left( 1+ \frac{1}{2}\right) 2^{2/p}\|\Gamma_{\mu}\|_{A^p}=3 \cdot 2^{2/p-1}\|\Gamma_\mu\|_{A^p}.
\end{split}
\]
Therefore 
\begin{align*}
  \|\Gamma_\mu\|_{A^p}&\geq\frac{1}{3\cdot 2^{2/p-1}}\int_{0}^{1}\frac{1}{(1-t)^{2/p}} \,d\mu(t)
  =A(p) \int_{0}^{1}\Theta_p(t) \,d\mu(t)  . 
\end{align*}
Theorem \ref{main theom boundedness} is now proved for $p\neq 2$.\\
\end{proof}
\vspace{11 pt}
\section{Proof of boundedness of $\Gamma_\mu$: $p=2$}
\noindent 
We start this section by computing $\|\Gamma_\mu(1)\|_{A^2}^2$. This quantity will be necessary for the conclusion of the proof of Theorem \ref{main theom boundedness}. We use some ideas inspired by \cite{Galanopoulos2010}.

\begin{lem}\label{L:estimate Gmu(1)}
Let $\mu$ be a probability measure on $[0,1)$. Then
\begin{equation*}\label{estimate Gammamu1p=2}
    \|\Gamma_\mu(1)\|_{A^2}^2\sim \int_0^1 \mu[0,t] \log\frac{e}{t} \,d\mu(t).
\end{equation*}
\end{lem}
\begin{proof}
We note that
\[
\Gamma_\mu(1)(z)=\int_0^1\frac{1}{1-(1-t)z}\,d\mu(t)=\sum_{n\geq 0}\int_{0}^1 (1-t)^n \,d\mu(t) \ z^n.
\]
Hence, by \eqref{norm bergman series}, we obtain that
\begin{align*}
\|\Gamma_\mu(1)\|_{A^2}^2=&\sum_{n\geq 0} \frac{\int_{0}^1 (1-t)^n \,d\mu(t)\int_{0}^1 (1-s)^n \,d\mu(s)}{n+1}\\
=&\int_0^1 \int_0^1 \sum_{n\geq 0}\frac{[(1-t)(1-s)]^n}{n+1}\,d\mu(t)\,d\mu(s)\\
=& \int_0^1 \int_0^1 1+\sum_{n\geq 1}\frac{\{(1-t)(1-s)\}^n}
{n+1}\,d\mu(t)\,d\mu(s)\\
\sim& \int_0^1 \int_0^t \log\frac{e}{1-(1-t)(1-s)} \,d\mu(s)\,d\mu(t).
\end{align*}
More precisely,
\[
\frac{1}{2} \norm{\Gamma_\mu(1)}_{A^p}^2  
\leq
\int_0^1 \int_0^t \log\frac{e}{1-(1-t)(1-s)} \,d\mu(s)\,d\mu(t)
\leq
\norm{\Gamma_\mu(1)}_{A^p}^2.  
\]
Consequently,
\begin{align*}
\|\Gamma_\mu(1)\|_{A^2}^2&\geq  \int_0^1 \int_0^t \log\frac{e}{1-(1-t)^2} \,d\mu(s)\,d\mu(t)\\
&=\int_0^1 \mu[0,t]\log\frac{e}{t(2-t)}\,d\mu(t)\\
&= \int_0^1 \mu[0,t] \log\frac{e}{t} \left( 1-  \frac{\log(2-t)}{\log\left(e/t\right)}\right)\,d\mu(t)\\
&\geq (1-\log(2))\int_0^1 \mu[0,t] \log\frac{e}{t} \,d\mu(t)
\end{align*}
and
\begin{align*}
\|\Gamma_\mu(1)\|_{A^2}^2&\leq 2  \int_0^1 \int_0^t \log\frac{e}{1-(1-t)} \,d\mu(s)\,d\mu(t)= 2 \int_0^1  \mu[0,t]\log\frac{e}{t}\,d\mu(t),
\end{align*}
which proves the statement of the lemma.\\
\end{proof}
\noindent
We are now ready to prove Theorem \ref{main theom boundedness} for $p=2$.
\begin{proof}[Proof of Theorem \ref{main theom boundedness} when $p=2$]
We recall that in Lemma \ref{lem-fond}, we have already verified that
$$
   \|\Gamma_\mu\|_{A^2}\geq   \int_0^1 \dfrac{1}{1-t}\,d\mu(t).
$$
Therefore, if $\Gamma_\mu$ is bounded in $A^2$, we obtain that
\begin{align*}
&\quad \int_0^1 \mu[0,t]\log\left(\frac{e}{t}\right)\,d\mu(t)+\left(\int_{0}^{1}\frac{1}{1-t}\,d\mu(t)\right)^{2}\\
&\lesssim \|\Gamma_\mu(1)\|_{A^2}^2 +  \|\Gamma_\mu\|_{A^2}^{2}
\lesssim \|\Gamma_\mu\|_{A^2}^2.
\end{align*}
On the other hand, by using Lemma \ref{L:estimates T_t} and Lemma \ref{L:estimate Gmu(1)}, we note that
\begin{small}
\begin{align*}
    \|\Gamma_\mu(f)\|_{A^2}^2\leq &2\|\Gamma_\mu(f(0))\|_{A^2}^2+2\|\Gamma_\mu(f-f(0))\|_{A^2}^2\\
    \leq &2|f(0)|^2 \|\Gamma_\mu(1)\|_{A^2}^2+2\left( \int_{0}^1\|T_t(f-f(0))\|_{A^2}\,d\mu(t)\right)^2\\
    \leq & 4 \|f\|_{A^2}^2 \int_0^1 \mu[0,t] \log\left(\frac{e}{t}\right) \,d\mu(t) +2\|S^*\|^2_{A^2} \|f\|_{A^2}^2\left(\int_0^1\frac{1}{1-t}\,d\mu(t)\right)^2.
\end{align*}
\end{small}
Therefore, by fixing $B(2)=\sqrt{2\max(2,\|S^*\|_{A^2}^2)}$, we have that 
\begin{align*}
\|\Gamma_\mu(f)\|_{A^2}\leq B(2)\|f\|_{A^2} \ \sqrt{\int_0^1 \mu[0,t]\log\left(\frac{e}{t}\right)\,d\mu(t)+\left(\int_0^1\frac{1}{1-t}\,d\mu(t)\right)^{2}}
\end{align*}
from which the estimate from above follows.\\
\end{proof}
\noindent 
First of all, we highlight that the the first term in the definition of $\Theta_2(t)$ is fundamental. Indeed, if $\mu=\delta_0$, then $\Gamma_{\delta_0}$ is unbounded in $A^2$ but 
$$
\int_{0}^1\dfrac{1}{1-t} \, d\delta_0(t)<\infty.
$$

\vspace{11 pt}
\noindent
It is worth mentioning that the operator $\Gamma_\mu$ is bounded from the closed subspace $A^{2}_{0}=\{f\in A^2: f(0)=0\}$ to $A^2$ if and only if
$$
\int_{0}^{1}\frac{1}{1-t}\,d\mu(t)<\infty. 
$$
This result is in accordance with the other values of $p$ and it shows that it is $\|\Gamma_{\mu}(1)\|_{A^2}$ that has an unexpected behavior. 
We also note that $\Gamma_\mu(f)(0)=0$ is seldom true, even if we choose $f \in A^2_0$.

\vspace{11 pt}
\noindent
When $p\not=2$, an application of the generalized Minkowski's inequality is enough to provide the "correct" estimate for the norm of   $\Gamma_{\mu}$. For this reason, it is tempting to consider the condition
\begin{equation}\label{wrong condition}
 \,\,\int_{0}^{1}\sqrt{\log(e/t)}\,d\mu(t)<\infty,  
\end{equation}
which captures the behavior of 
\begin{equation*}\label{wrong condition2}
 \,\,\int_{0}^{1}\frac{\sqrt{\log(e/t)}}{1-t}\,d\mu(t) 
\end{equation*}
when $t$ is close to $0$, see Lemma \ref{L:estimates T_t}, instead of
\begin{equation*}
\sqrt{\int_{0}^{1}\mu[0,t]\log(e/t)\,d\mu(t)}<\infty.
\end{equation*}
However,  \eqref{wrong condition} is not a necessary condition for the boundedness of $\Gamma_\mu$ on $A^2$.

\begin{prop}\label{Prop_counter}
There exists a positive, absolutely continuous measure $\mu$ such that \eqref{wrong condition} is not satisfied but $\Gamma_\mu$ is bounded on $A^2$.
\end{prop}
\begin{proof}
We consider the measure
$$
d\mu(t)=C\left(\frac{1}{2} \frac{1}{\log\left(\frac{e}{t}\right)^{\frac{3}{2}}}\frac{1}{\log\left(\log\left(\frac{e}{t}\right)\right)}\frac{1}{t}+ \frac{1}{\log\left(\frac{e}{t}\right)^{\frac{3}{2}}}\frac{1}{\log\left(\log\left(\frac{e}{t}\right)\right)^2}\frac{1}{t}\right)\chi_{[0,\frac{1}{2}]}(t) \, dt,
$$
where 
$$
C=\log \left( 2e\right)^\frac{1}{2}\log\left(\log 2e\right).
$$
It follows that 
$$
\mu[0,t]=\begin{cases}
C\log\left(\frac{e}{t}\right)^{-\frac{1}{2}}\frac{1}{\log\left(\log\left( \frac{e}{t}\right)\right)},\qquad &\text{if } 0<t\leq \frac{1}{2};\\[0.2in]
1, \qquad &\text{if } \frac{1}{2}<t\leq 1 .
\end{cases}
$$
We note that 
\begin{align*}
\int_{0}^{1}\sqrt{\log(e/t)}\,d\mu(t)> \frac{C}{2}\int_0^{\frac{1}{2}} \frac{1}{\log\left(\frac{e}{t}\right)}\frac{1}{\log\left(\log\left(\frac{e}{t}\right)\right)}\frac{1}{t} \, dt = \infty,         
\end{align*}
even if $\Gamma_\mu$ is bounded on $A^2$ since the condition of Theorem \ref{main theom boundedness} is satisfied. Indeed 
\begin{align*}
&\sqrt{\int_0^1 \mu[0,t]\log\left(\frac{e}{t}\right)\,d\mu(t)+\left(\int_0^1\frac{1}{1-t}\,d\mu(t)\right)^{2}}\\
&\quad \leq \sqrt{ C^2\int_0^{\frac{1}{2}} \frac{1}{t}\frac{1}{\log\left(\frac{e}{t}\right)}\frac{1}{\log\left(\log\left( \frac{e}{t}\right)\right)^2}\left(1+\frac{1}{\log\left(\log\left( \frac{e}{t}\right)\right)}\right)  +4} <\infty .        
\end{align*}
\end{proof}
\CB
\noindent
Before concluding this section, we provide another necessary condition for the boundedness of $\Gamma_\mu$ on $A^2$. To formulate it, we consider the adjoint of $\Gamma_\mu$ acting on the classical Dirichlet space $\mathcal{D}$.

\begin{lem}\label{the adjoint Gm}
Let $p=2$. If $\Gamma_\mu$ is bounded on $A^2$, then the adjoint operator of $\Gamma_{\mu}$ on $A^2$, denoted by $\Gamma_{\mu}^{*}$, is given by
$$
\Gamma_{\mu}^{*}(f)(z)=\int_{0}^{1}T_{t}^{*}(f)(z)\,d\mu(t)=\int_{0}^{1}T_{1-t}(f)(z)\,d\mu(t) 
$$
for every $f$ in the Dirichlet space $\mathcal{D}$.
\end{lem}
\begin{proof}
First of all if $f(z)=\sum_{n}a_nz^n$ and $g(z)=\sum_n b_nz^n$ belong to $A^2$, then
$$
\left\langle f,g\right\rangle=\sum_n \dfrac{a_n\overline{b_n}}{n+1}=\sum_{n} a_n \overline{c_n}=\left\langle f,G\right\rangle_c
$$
where $c_n=b_n/(n+1)$ and $G(z)=\sum_n c_n z^n$ belongs to the Dirichlet space $\mathcal{D}$. Since also the reverse inclusion holds, we can identify the dual of $A^2$ with $\mathcal{D}$ by using the Cauchy pairing. 

\vspace{11 pt}
\noindent 
We first assume that $f(z)$ and $G(z)$ are polynomials. By using Proposition 1 of \cite{Bellavita2023} and Fubini's Theorem, we obtain that
\begin{align*}
    \left\langle \Gamma_\mu(f),G\right\rangle_c &=
    \int_{\mathbb T} \Gamma_\mu (f) (\zeta) \overline{G(\zeta)}|d\zeta |
    =\,\int_{0}^1 \int_{\mathbb T} T_t (f) (\zeta)\overline{G(\zeta)}| d\zeta |  \,d\mu(t)\\
     &=\, \int_{0}^1 \left\langle T_t(f), G\right\rangle_c  \,d\mu(t)=\, \int_{0}^1 \left\langle f, T_{1-t}(G)\right\rangle_c  \,d\mu(t)\\
    &=\, \int_{\mathbb T}f(\zeta) \overline{ \left( \int_{0}^1 {T_{1-t} (G) (\zeta)} \,d\mu(t) \right) } | d\zeta | = \left\langle f,\Gamma^*_\mu\left(G\right)\right\rangle_c .
\end{align*}
Since the partial Taylor sums are dense in $A^2$ and $\mathcal{D}$, we have that
\begin{align*}
    \left\langle \Gamma_\mu(f),G\right\rangle_c &=\lim_{n \to \infty} \left\langle \Gamma_\mu(S_n(f)),S_n(G)\right\rangle_c\\
    &=\lim_{n \to \infty} \left\langle S_n(f),\Gamma^*_\mu(S_n(G))\right\rangle_c\\
    &= \left\langle f,\Gamma^*_\mu(G)\right\rangle_c,
\end{align*}
which concludes the proof.\\
\end{proof}

\begin{prop}
Let $p=2$. If $\Gamma_{\mu}$ is bounded on $A^2$, then for every $0<a<1$, we have that
$$
\int_{0}^1\left(\log\frac{e}{t}\right)^{a/2}\,d\mu(t)<\infty.
$$    
\end{prop}
\begin{proof}
We note that $f_a(z):=\log\left(e/(1-z)\right)^{a/2} \in \mathcal{D}$. Therefore, according to Lemma \ref{the adjoint Gm},
\begin{align*}
 \|\Gamma_\mu\|_{A^2}&\sim\|\Gamma_\mu^*\|_{\mathcal{D}}\geq \frac{1}{\|f_a\|_{\mathcal{D}}}\|\Gamma^*_\mu(f_a)\|_{\mathcal{D}}\\
 &\geq \frac{1}{\|f_a\|_{\mathcal{D}}}|\Gamma^*_\mu(f_a)(0)|\\
  &=\frac{1}{\|f_a\|_{\mathcal{D}}}\int_{0}^1\left(\log\frac{e}{t}\right)^{a/2}\,d\mu(t),
\end{align*}
from which the statement follows.\\
\end{proof}

\noindent
We point out that, by monotone convergence, the condition
$$
\lim_{a \to 1} \int_{0}^1\left(\log\frac{e}{t}\right)^{a/2}\,d\mu(t)<\infty 
$$ is exactly \eqref{wrong condition}. Even if we can say that for every $a<1$ this integral is finite, we do not have an uniform bound, since the quantity $1/{\|f_a\|_{\mathcal{D}}}$ is going to zero.
\vspace{11 pt}
\section{Proof of compactness}
\noindent In Lemma \ref{lem-fond}, we have already estimated the essential norm of $\Gamma_\mu$ from below. Indeed, let $X$ be a reflexive space. If $\{w_n\}\subset X$ is a unitary weakly null sequence, then
\begin{equation}\label{Essential norm below}
\begin{split}
      \|T\|_{e,X} &=\inf_{K} \|T-K\|_{X} \\
      &\geq \inf_K\lim_{n \to \infty}\|T(w_n)-K(w_n)\|_{X}\\
      &\geq \inf_K\lim_{n \to \infty} \left\vert  \|T(w_n)\|_{X}-\|K(w_n)\|_{X}\right\vert= \lim_{n \to \infty } \|T(w_n)\|_{X} .
\end{split} 
 \end{equation}
 
\noindent For the estimate from above of $\|\Gamma_\mu\|_{e,A^p}$, we use Lemma 3.2 of \cite{Lindstrom2022} (see also \cite{Lindstrm2024C}). For the sake of completeness, we state it here.
\begin{lem}\label{aprox Lemm}
Let $1 < p < \infty$. There
exists a sequence of compact operators $\{L_n\}_n$ such that
$$
\limsup_{n}\|I-L_n\|_{A^p} \leq 1.
$$ Moreover, for every $0 <R< 1$, we have
$$
\limsup_{n\to \infty}\sup_{\|f\|_{A^{p}}=1}\sup_{|z|\leq R}|(I-L_n)(f)(z)| = 0.
$$
\end{lem}

\noindent 
We are now ready to compute the essential norm of $\|\Gamma_\mu\|_{e,A^p}$ when $1<p<\infty$. 
\begin{proof}[Proof of Theorem \ref{Compact p>1}]
By Lemma \ref{lem-fond} and  \eqref{Essential norm below}, we know that
\begin{equation*}
    \|\Gamma_\mu\|_{e,A^p}\geq \lim_{a \to 2/p}\|\Gamma_\mu(f_a/\|f_a\|_{A^p})\|_{A^p}=\int_{0}^1 \frac{t^{2/p-1}}{(1-t)^{2/p}}\,d\mu(t)\ .
\end{equation*}

 \vspace{11 pt}\noindent 
 To complete the proof of the theorem, we only need to estimate the essential norm of $\Gamma_\mu$ from above. 

\vspace{11 pt}\noindent 
Let $D_{R,t}=\varphi_{t}(\mathbb{D})\cap \overline{D(0,R)}$ and $D_{R,t}^{c}:=\varphi_{t}(\mathbb{D})\setminus D(0,R)$, where $0<R<1$. Then
\begin{align*}
\|\Gamma_{\mu}(f)\|_{A^{p}}&\leq\|\Gamma_\mu(f(0))\|_{A^p}+\|\Gamma_\mu(f-f(0))\|_{A^p}\leq |f(0)|\|\Gamma_\mu(1)\|_{A^p}\\
& + \bigintsss_{0}^{1}\frac{t^{2/p-1}}{(1-t)^{2/p}}\left(\int_{\varphi_{t}(\mathbb{D})}|w|^{p-4}|(f-f(0))(w)|^{p}\, dA(w)\right)^{1/p}\,d\mu(t)\\
&= I+II .
\end{align*}
We name $f(z)-f(0)=zg(z)$. For the second quantity, we have that
\begin{align*}
II&\leq\int_{0}^{1}\frac{t^{2/p-1}}{(1-t)^{2/p}}\biggl[\left(\int_{D_{R,t}}|w|^{p-4}|wg(w)|^{p}\, dA(w)\right)^{1/p}\\
&\qquad \qquad +\left(\int_{D_{R,t}^{c}}|w|^{p-4}|wg(w)|^{p}\, dA(w)\right)^{1/p}\biggr]\,d\mu(t)\\
&\leq\int_{0}^{1}\frac{t^{2/p-1}}{(1-t)^{2/p}}
\biggl[\sup_{|w|\leq R}|g(w)|\left(\int_{\mathbb{D}} |w|^{2p-4}\, dA(w)\right)^{1/p}\\
&\qquad \qquad +\sup_{w \in D_{R,t}^{c}}|w|^{1-4/p}\cdot\|f-f(0)\|_{A^{p}}\biggr]\,d\mu(t)\\
&\leq\int_{0}^{1}\frac{t^{2/p-1}}{(1-t)^{2/p}}
\biggl[ \sup_{|w|\leq R}|g(w)|C_p+\sup_{w \in D_{R,t}^{c}}|w|^{1-4/p}\cdot\|f\|_{A^{p}}\\
&\qquad \qquad +\sup_{w \in D_{R,t}^{c}}|w|^{1-4/p}\cdot|f(0)|\biggr]\,d\mu(t) .
\end{align*}
Let $\{L_n\}_n$ be the sequence of compact operators as described in Lemma \ref{aprox Lemm}.
Then
\begin{equation*}
\begin{split}
\norm{\Gamma_\mu}_{e,A^p} &\leq  \limsup_{n\to\infty} \sup_{\|f\|_{A^p}=1}\norm{(\Gamma_\mu   -  \Gamma_\mu L_n  ) 
(f  )}_{A^p}\\
& = \limsup_{n\to\infty} \sup_{\|f\|_{A^p}=1} \norm{\Gamma_\mu( (I-L_n) (f  )}_{A^p}
\end{split}
\end{equation*}
and the last expression is smaller than
\begin{small}
\begin{equation*}\label{eq:hmmp=2}
\begin{split}
    & \leq \limsup_{n\to\infty} \sup_{\|f\|_{A^p}=1}|(I-L_n) (f)(0)|\left( \|\Gamma_\mu(1)\|_{A^p}+\max\{ 1, R^{1-4/p}  \}   \int_{0}^1 \frac{t^{2/p-1}}{(1-t)^{2/p}}\,d\mu(t) \right)\\
&\quad \qquad + \limsup_{n\to\infty} \sup_{\|f\|_{A^p}=1}\sup_{|{w}|= R}   |{S^*\circ(I-L_n)(f)(w)}|C_p \int_0^1 \frac{t^{2/p-1}}{(1-t)^{2/p}}  \,d\mu(t)   \\
&\quad  \qquad +\limsup_{n\to\infty} \sup_{\|f\|_{A^p}=1} \norm{ ( I-L_n) f  }_{A^p} \max\{ 1, R^{1-4/p}  \}   \int_0^1 \frac{t^{2/p-1}}{(1-t)^{2/p}}   \,d\mu(t) .
\end{split} 
\end{equation*} 
\end{small}
The first term tends to zero because of \eqref{aprox Lemm}. Moreover
\begin{small}
\begin{align*}
    \sup_{|{w}|= R}   |{S^*\circ(I-L_n)(f)(w)}|  &\leq \frac{1}{R}\left( \sup_{|{w}|\leq R}|(I-L_n)(f)(w)|+|(I-L_n)(f)(0)|\right) .
\end{align*}
\end{small}
Therefore, using once more Lemma \ref{aprox Lemm},
\begin{align*}
&\limsup_{n\to\infty} \sup_{\|f\|_{A^p}=1}\max\{ 1, R^{1-4/p}  \}    \sup_{|{w}|\leq R}|(I-L_n)(f)(w)|\\
&\qquad \qquad +\limsup_{n\to\infty} \sup_{\|f\|_{A^p}=1}\max\{ 1, R^{1-4/p}  \}   |(I-L_n)(f)(0)|=0.
\end{align*}
Consequently, by the boundedness of $\Gamma_\mu$ and Theorem \ref{main theom boundedness}, we have
\begin{align*}
    \norm{\Gamma_\mu}_{e} &\leq \limsup_{n\to\infty} \sup_{\|f\|_{A^p}=1} \norm{ ( I-L_n) f  }_{A^p} \max\{ 1, R^{1-4/p}  \}   \int_0^1 \frac{t^{2/p-1}}{(1-t)^{2/p}}   \,d\mu(t),
\end{align*}
and letting $R\to1$, we obtain the desired upper estimate. \\
\end{proof}
\noindent 
Next, we move to the case $p=1$. We need the following preliminary lemma.
\begin{lem}\label{Lem aux}
Let $\{f_n\}\subset A^1$ be a weakly null convergent sequence. Then, for every fixed $0\leq t<1$, $T_t(f_n)$ is strongly null convergent.    
\end{lem}
\begin{proof}
Since $\{f_n\}$ is weakly null, $f_n$ converge to zero on every compact subset of $\mathbb D$. Moreover, $\varphi_t(\mathbb D)\subset \mathbb D$ touches $\partial \mathbb D$ only at $1$. Consequently, for every $0\leq t<1$, $f_n\circ\varphi_t$ converge in measure to zero and, due to \cite[p. 295]{dunford1988}, $\{f_n\circ\varphi_t\}$ is also strongly null. The Lemma is proved since the multiplication by $w_t$ does not change the behaviour of $f_n\circ \varphi_t$.\\  
\end{proof}

\begin{proof}[Proof of Theorem \ref{Compact p=1}]


Let $f_{a}(z)=1/(1-z)^{a}$ as in Lemma \ref{lem-fond}. We consider the bounded sequence 
$$
\left\lbrace\dfrac{f_a}{\|f_a\|_{A^1}}\right\rbrace_{1<a<2}\ ,
$$
which is converging to zero uniformly on compact subsets of $ \mathbb D$. Then, if $\Gamma_\mu$ were compact, \cite[Lemma 3.7]{Tjani2003} would imply that $\Gamma_\mu(f_a/\|f_a\|_{A^1})$ would go to zero. However, in Lemma \eqref{lem-fond}, we have verified that
\[
\lim_{a \to 2}\|\Gamma_\mu(f_a/\|f_a\|_{A_1})\|_{A^1}\geq \int_{0}^1\frac{t}{(1-t)^2}\,d\mu(t).
\]
Consequently, if $\mu \neq \delta_0$, then $\Gamma_\mu(f_a/\|f_a\|_{A^1})$ is not strongly converging to zero and thus $\Gamma_\mu$ cannot be compact. On the other hand, if $\mu=\delta_0$, then $\Gamma_\mu$ is compact since it is a rank $1$ operator.

\vspace{11 pt}
\noindent 
In order to show that $\Gamma_\mu$ is completely continuous, we use Lemma \ref{Lem aux} which states that $T_t$ is completely continuous for every $t\in [0,1)$. Consider now a sequence of functions $\{ f_n \}$ which is weakly null in $A^1$. Then, by the complete continuity of $T_t$ we have that $\lim_n \Vert T_t f_n \Vert_{A^1} = 0$, for all $0\leq t<1 $. Furthermore, using Lemma \ref{L:estimates T_t}, we have
  \begin{equation*}
      \Vert T_t  f_n \Vert_{A^1} \leq \sup_{n} \Vert f_n \Vert_{A^1} \Vert  T_t \Vert_{A^1} \,\leq \,C(1)\, \sup_n \Vert f_n \Vert_{A^1} \frac{1}{(1-t)^2}.
  \end{equation*}
 By applying the dominated convergence theorem together with Theorem \ref{main theom boundedness}, we conclude that 
  \[
  \limsup_n \Vert \Gamma_\mu f_n \Vert_{A^1} \leq \limsup_n \int_0^1 \Vert T_t f_n \Vert_{A^1} \,d\mu(t) = 0.
  \]\\
  \end{proof}
\section{Acknowledgments}
\noindent 
 The authors would like to express their gratitude to professor Aristomenis Siskakis for introducing them to this problem and for all the valuable discussions about this topic.

\vspace{11 pt}
\noindent 
The first author is a member of Gruppo Nazionale per l’Analisi Matematica, la Probabilit\`a e le loro Applicazioni (GNAMPA) of Istituto Nazionale di Alta Matematica (INdAM) and he was partially supported by PID2021-123405NB-I00 by the Ministerio de Ciencia e Innovaci\'on and by the Departament de Recerca i Universitats, grant 2021 SGR 00087.\\
The first, the second and the fifth  authors were partially supported by the Hellenic Foundation for Research and Innovation (H.F.R.I.) under the '2nd Call for H.F.R.I. Research Projects to support Faculty Members \& Researchers' (Project Number: 4662).\\ 
The third and the sixth authors were supported by Engineering and Physical Sciences Research Council grant EP/X024555/1.\\
The fourth author was supported by the Magnus Ehrnrooth Foundation.\\
The sixth author was supported by Engineering and Physical Sciences Research Council grant grant EP/Y008375/1.\\

\section{Declaration}
\noindent
The authors state no conflict of interest. No data-set have been used.

\bibliographystyle{amsplain}
\bibliography{bibliography}

\end{document}